\documentclass[12pt]{article}%
   
\usepackage{amsmath,enumerate}
\usepackage{amsfonts}
\usepackage{amssymb}

\newtheorem{theorem}{Theorem}[section]

\newtheorem{lemma}[theorem]{Lemma}

\newtheorem{problem}[theorem]{Problem}

\newenvironment{proof}[1][Proof]{\noindent\textbf{#1.} }
{\hfill \ \rule{0.5em}{0.5em}}

\title{The Zarankiewicz problem in 3-partite graphs}

\date{}

\author{	Michael Tait\thanks{Department of Mathematical Sciences, Carnegie Mellon University, 
\texttt{mtait@cmu.edu}. Research is supported by NSF grant DMS-1606350}
	\and
	Craig Timmons\thanks{Department of Mathematics and Statistics, 
	California State University Sacramento, \texttt{craig.timmons@csus.edu}.
Research supported in part by Simons Foundation Grant \#359419.}
	}

\begin{document}

\maketitle

\begin{abstract}
Let $F$ be a graph, $k \geq 2$ be an integer, and write $\textup{ex}_{ \chi \leq k } (n ,  F)$ for the 
maximum number of edges in an $n$-vertex graph that is $k$-partite and has no subgraph isomorphic 
to $F$.  The function $\textup{ex}_{ \chi \leq 2} ( n , F)$ has been studied by many researchers.
Finding $\textup{ex}_{ \chi \leq 2} (n , K_{s,t})$ is a 
special case of the Zarankiewicz problem.  
We prove an analogue of the K\"{o}v\'{a}ri-S\'{o}s-Tur\'{a}n Theorem for 3-partite graphs 
by showing 
\[
\textup{ex}_{ \chi \leq 3} (n , K_{s,t} ) \leq \left( \frac{1}{3} \right)^{1 - 1/s} \left( \frac{ t - 1}{2}  + o(1) \right)^{1/s} n^{2 - 1/s}
\]
for $2 \leq s \leq t$.  
Using Sidon sets constructed by Bose and Chowla, we prove that this upper bound is asymptotically best 
possible in the case that $s = 2$ and $t \geq 3$ is odd, i.e.,
$\textup{ex}_{ \chi \leq 3} ( n , K_{2,2t+1} ) = \sqrt{ \frac{t}{3}} n^{3/2} + o(n^{3/2})$
for $t \geq 1$.  
In the cases of $K_{2,t}$ and $K_{3,3}$, we use a result of Allen, Keevash, Sudakov, and Verstra\"{e}te,
to show that a similar upper bound holds for all $k \geq 3$, 
and gives a better constant when $s=t=3$.  Lastly, we point out an interesting
 connection between difference families from design theory and $\textup{ex}_{ \chi \leq 3 } (n ,C_4)$. 
\end{abstract}


\section{Introduction}

Let $G$ and $F$ be graphs.  We say that $G$ is \emph{$F$-free} if $G$ does not contain 
a subgraph that is isomorphic to $F$. 
The \emph{Tur\'{a}n number of $F$} is the maximum 
number of edges in an $F$-free graph with $n$ vertices.  
This maximum is denoted $\textup{ex}(n , F)$.  An 
$F$-free graph with $n$ vertices and $\textup{ex}(n , F)$ edges is 
called an \emph{extremal graph} for $F$.  One of the most well-studied cases
is when $F = C_4$, a cycle of length four.  This problem was considered by 
Erd\H{o}s \cite{1936} in 1938.
While this arose as a problem in extremal graph theory,
the best constructions come from finite geometry and use projective planes 
and difference sets.  
Roughly 30 years later, 
Brown \cite{brown}, and 
Erd\H{o}s, R\'{e}nyi, and S\'{o}s \cite{erdos-renyi, erdos-renyi-sos}
independently showed that $\textup{ex}(n , C_4) = \frac{1}{2} n^{3/2} + o(n^{3/2})$.
They constructed, for each prime power $q$, a $C_4$-free 
graph with $q^2 + q + 1$ vertices and $\frac{1}{2}q(q+1)^2$ edges.  
These graphs are examples of orthogonal polarity graphs 
which have since been studied and applied to 
other problems in combinatorics.
Answering a question of Erd\H{o}s, F\"{u}redi \cite{furedi-1, furedi-2} showed 
that for $q > 13$, orthogonal polarity graphs 
are the only extremal graphs for $C_4$ when the number 
of vertices is $q^2 + q +1$.  F\"{u}redi \cite{F3} also
used finite fields to construct, for each $t \geq 1$, $K_{2,t+1}$-free 
graphs with $n$ vertices and $\sqrt{ \frac{t}{2} } n^{3/2} + o(n^{3/2})$ edges.  
This construction, together with the famous upper bound of
K\"{o}v\'{a}ri, S\'{o}s, and Tur\'{a}n \cite{KST}, shows that 
$\textup{ex}(n , K_{2,t+1} ) = \sqrt{ \frac{t}{2} } n^{3/2} + o(n^{3/2})$ for all $t \geq 1$.

Because of its importance in extremal graph theory, variations of
 the bipartite Tur\'{a}n problem have been considered.   
One such instance is to find the maximum number of 
edges in an $F$-free $n \times m$ bipartite graph.
Write $\textup{ex}( n , m , F)$ for this maximum.
Estimating $\textup{ex}(n , n , K_{s,t})$ 
is the ``balanced'' case of the Zarankiewicz problem.  
Recall that the Zarankiewicz problem is to find $z(m,n,s,t)$, which 
is the maximum number of 1's in an $m \times n$ 0-1 matrix with 
no $s \times t$ submatrix of all 1's.  The best known upper bound 
on $z(m,n,s,t)$ was proved by Nikiforov \cite{nikiforov} who showed 
\[
z(m,n,s,t) \leq ( s- t +1)^{1/t} n m^{1 -1/t} + (t-1) m^{2-2/t} + ( t- 2)n
\]
for $s \geq t$.  This improved an earlier bound of F\"{u}redi \cite{furedi z}
in the lower order terms.  
When $m = n,$ one can observe that $z(n,n,s,t) =\textup{ex}(n,n,K_{s,t})$.  
The results of \cite{F3, KST} show
that $\textup{ex}(n , n , K_{2,t+1} ) = \sqrt{t} n^{3/2} + o(n^{3/2})$ for $t \geq 1$.
The case when $F$ is a cycle of even length has also received considerable attention. 
Naor and Verstra\"{e}te \cite{naor-verstraete} studied the case when $F = C_{2k}$.  
More precise estimates were obtained by F\"{u}redi, Naor, and Verstra\"{e}te  \cite{furedi-naor-verstraete}
when $F = C_6$.  
For more results along these lines, see \cite{decaen-szekely-1992, decaen-szekely-1997, gyori-1997}
and the survey of F\"{u}redi and Simonovits \cite{furedi-simonovits} to name a few.

Now we introduce the extremal function that is the focus of this paper.
For an integer $k \geq 2$, define 
\[
\textup{ex}_{ \chi \leq k } ( n , F)
\]
to be the maximum number of edges in an $n$-vertex graph $G$ that is $F$-free and has 
chromatic number at most $k$.  
Thus, $\textup{ex}_{ \chi \leq 2 } (n , F)$ is the maximum number of 
edges in an $F$-free bipartite graph with $n$ vertices (the part sizes need not be the same).    
Trivially, 
\[
\textup{ex}_{ \chi \leq k } ( n , F) \leq \textup{ex}(n , F)
\]
for any $k$.  
In the case that $k = 2$, 
\[
\textup{ex}_{ \chi \leq 2} ( n , K_{2,t} ) = \frac{\sqrt{t-1}}{ 2 \sqrt{2}} n^{3/2} + o (n^{3/2})
\]
by \cite{F3, KST}.  
Our focus will be on $\textup{ex}_{ \chi \leq 3} (n , K_{2,t})$ and our  
first result gives an 
upper bound on $\textup{ex}_{ \chi \leq 3 } ( n , K_{s,t})$.

\begin{theorem}\label{chi 3 ub}
For $n \geq 1$ and $2 \leq s \leq t$, 
\[
\textup{ex}_{ \chi \leq 3} ( n , K_{s,t}) \leq \left( \frac{1}{3} \right)^{1 - 1/s} \left(  \frac{t-1}{2}  + o(1) \right)^{1/s} n^{2-1/s}.
\]
\end{theorem}

When $s = 2$, Theorem \ref{chi 3 ub} improves the trivial bound 
\[
\textup{ex}_{ \chi \leq 3 } ( n , K_{2,t} ) \leq \textup{ex}(n , K_{2,t}) = \frac{\sqrt{t-1}}{2} n^{3/2} + o (n^{3/2}).
\]

Allen, Keevash, Sudakov, and Verstra\"{e}te \cite{aksv} constructed 3-partite graphs with $n$ vertices that 
are $K_{2,3}$-free and have $\frac{1}{\sqrt{3} } n^{3/2} - n$ edges.  
This construction shows that Theorem \ref{chi 3 ub} is asymptotically best possible in the case that $s=2$, $t = 3$.
Our next theorem, which is the main result of this paper, shows that 
Theorem \ref{chi 3 ub} is, in fact, asymptotically best possible for $s =2$ and all odd integers $t \geq 3$.

\begin{theorem}\label{main theorem}
For any integer $t \geq 1$,
\[
\textup{ex}_{ \chi \leq 3 } ( n , K_{2 , 2t+1} ) = \sqrt{ \frac{t}{3} } n^{3/2} + o( n^{3/2} ).
\]
\end{theorem}

We believe that the most interesting remaining open case is determining the behavior when forbidding $K_{2,2} = C_4$.
\begin{problem}\label{chi 3 problem}
Determine the asymptotic behavior of 
\[
\textup{ex}_{ \chi \leq 3 } ( n , C_4 ).
\]
\end{problem}

In particular it would be very interesting to know whether or not 
$\textup{ex}_{ \chi \leq 2 } ( n , C_4) \sim \textup{ex}_{ \chi \leq 3 } ( n , C_4)$. 
In Section \ref{conclusion}, we use a difference family from design theory to show 
that $\textup{ex}_{ \chi \leq 3} ( 123,C_4) = 615$, where the upper bound 
is a consequence of the counting argument used to prove Theorem \ref{chi 3 ub}.  
For comparison, $\textup{ex}_{\chi \leq 2} ( 123 ,C_4) \leq 521$.  
We discuss this further in Section \ref{conclusion}.

In the special cases $s =2 , t \geq 2$ and $s =t =3 $, we can use a lemma of 
 Allen, Keevash, Sudakov, and Verstra\"{e}te \cite{aksv} to
 prove an upper bound on $\textup{ex}_{ \chi \leq k } ( n , K_{s,t})$ that holds for any $k \geq 3$. 
This argument gives a better constant than the one provided by Theorem \ref{chi 3 ub} when $s=t=3$. 

\begin{theorem}\label{new ub}
Let $k \geq 3$ be an integer.
For any integer $t \geq 2$,  
\[
\textup{ex}_{ \chi \leq k } (n , K_{2,t}) \leq \left(  \left( 1 - \frac{1}{k} \right)^{1/2} + o(1) \right) \frac{ \sqrt{t-1} }{2} n^{3/2}. 
\]
Also,
\[
\textup{ex}_{\chi \leq k} ( n , K_{3,3} ) \leq \left(  \left( 1 - \frac{1}{k} \right)^{2/3} + o(1) \right) \frac{n^{5/3} }{2}.
\]
\end{theorem} 

A random partition into $k$ parts of an $n$-vertex $K_{2,t}$-free graph with $\frac{\sqrt{t-1}}{2} n^{3/2} + o (n^{3/2} )$ edges 
gives a lower bound of 
\[
\textup{ex}_{ \chi \leq k } ( n , K_{2,t}) \geq \left( 1 - \frac{1}{k } \right) \frac{ \sqrt{t-1} }{2}n^{3/2} - o(n^{3/2} ).
\]
Similarly,
\[
\textup{ex}_{ \chi \leq k } (n , K_{3,3} ) \geq \left( 1 - \frac{1}{k } \right) \frac{ n^{5/3}  }{2}  - o(n^{5/3} ).
\]
We would like to remark that the lemma of Allen et.\ al.\ can be used to prove 
a more general version of Theorem \ref{new ub}.
Following \cite{aksv}, a family $\mathcal{F}$ of bipartite 
graphs is \emph{smooth} if there are real numbers 
$1 \leq \beta < \alpha < 2 $ and $\rho \geq 0$ such that 
\[
z(m,n, \mathcal{F}) = \rho m n^{\alpha -1} + O (n^{\beta} )
\]
for all $m \leq n$.  Here 
$z(m,n, \mathcal{F})$ is the maximum number
of edges in an $\mathcal{F}$-free $m \times n$ bipartite graph.
The graphs $K_{2,t}$ and $K_{3,3}$ are smooth.  
Another example of a smooth family is given in \cite{aksv}.
Under the smoothness hypothesis, Allen et.\ al.\ proved the following important
 result in the theory of bipartite Tur\'{a}n numbers, and made progress
 on a difficult conjecture of Erd\H{o}s and Simonovits.
 
 \begin{theorem}[Allen, Keevash, Sudakov, Verstra\"{e}te]\label{aksv theorem}
Suppose that $\mathcal{F}$ is a family of graphs that is $(\alpha ,\beta)$-smooth  
where $2 > \alpha > \beta \geq 1$.
There is a $k_0$ such that if $k$ is an odd integer with $k \geq k_0$ the following holds: every
extremal $\mathcal{F} \cup \{ C_k \}$-free family of graphs is near-bipartite.
\end{theorem}
For a more precise description of what is meant by near-bipartite, we refer the reader to
\cite{aksv}.  Roughly speaking, it means that one can remove a negligible 
number of edges from an extremal $\mathcal{F} \cup \{ C_k \}$-free graph
to make it bipartite.  
One of the keys to the proof of the Allen-Keevash-Sudakov-Verstra\"{e}te Theorem
was their Lemma 4.1.  This lemma allows one to transfer the density of an $\mathcal{F}$-free graph
to the density of a reduced graph obtained by
applying Scott's Sparse Regularity Lemma \cite{scott}.  Using Lemma 4.1 of \cite{aksv}, 
one can prove a version of Theorem \ref{new ub} for any family of 
bipartite graphs that is known to be smooth.
 
In the next section we prove Theorem \ref{chi 3 ub} and Theorem \ref{new ub}.  
In Section \ref{section lb} we prove Theorem \ref{main theorem}.  In Section \ref{conclusion},
we highlight the connection between $\textup{ex}_{ \chi \leq 3 } (n , C_4 )$ 
and difference families from design theory.


\section{Proof of Theorem \ref{chi 3 ub}}\label{section ub}

In this section we prove Theorem \ref{chi 3 ub}.  The proof is based on the 
standard double counting argument of K\"{o}v\'{a}ri, S\'{o}s, and Tur\'{a}n \cite{KST}.

\bigskip

\begin{proof}[Proof of Theorem \ref{chi 3 ub}]
Let $G$ be an $n$-vertex 3-partite graph that is $K_{s,t}$-free.  Let $A_1$, $A_2$, and $A_3$ be the parts of $G$. Define 
$\delta_i$ by $\delta_i n = |A_i|$.

By the K\"{o}v\'{a}ri-S\'{o}s-Tur\'{a}n Theorem \cite{KST}, there is a constant $\beta_{s,t} > 0$ such that 
the number of edges with one end point in $A_1$ and the other in $A_2$ is at most 
$\beta_{s,t} n^{2-1/s}$.  If there are $o(n^{2-1/s})$ edges between $A_1$ and $A_2$, then we may remove these 
edges to obtain a bipartite graph $G'$ that is $K_{s,t}$-free which gives
\[
e(G) \leq e(G' ) -  o( n^{2-1/s}) \leq \textup{ex}_{ \chi \leq 2} ( n , K_{s,t}).
\]
In this case, we may apply the upper bound of  F\"{u}redi \cite{furedi z} (or Nikiforov \cite{nikiforov}) to 
see that the conclusion of Theorem \ref{chi 3 ub} holds.  
Therefore, we may assume that there is a positive constant $c_{1,2}$ so that the number of 
edges between $A_1$ and $A_2$ is $c_{1,2} n^{2-1/s}$.  Similarly, let $c_{1,3} n^{2-1/s}$ and 
$c_{2,3} n^{2-1/s}$ be the number of edges between $A_1$ and $A_3$, and between $A_2$ and 
$A_3$, respectively.  

For a positive real number $x$, define 
\[
\binom{x}{s}
=
\left\{
\begin{array}{ll}
\frac{  x (x-1) (x-2) \cdots (x- s +1)  }{s!} & \mbox{if $x \geq s-1$,} \\
0 & \mbox{otherwise.}
\end{array}
\right.
\]
The function $f(x) = \binom{x}{s}$ is then a convex function.  
Using the assumption that $G$ is $K_{s,t}$-free and Jensen's Inequality, we have  
\begin{eqnarray}
(t-1)\binom{ |A_1| }{s} & \geq & \sum_{v \in A_2} \binom{ d_{A_1} (v) }{s} + \sum_{v \in A_3} \binom{ d_{ A_1} (v) }{s} 
\label{special inequality} \\
& \geq & 
|A_2| \binom{ \frac{1}{ |A_2 |} e( A_1 , A_2) }{s} + |A_3| \binom{ \frac{1}{ |A_3 |} e( A_1 , A_3) }{s} \nonumber \\
 & \geq &
 \frac{\delta_2 n}{s!}  \left( \frac{ e(A_1 , A_2) }{ |A_2 | } - s \right)^s + 
 \frac{\delta_3 n}{s!}  \left( \frac{ e(A_1 , A_3) }{ |A_3 | } - s \right)^s. \nonumber
  \end{eqnarray} 
 After some simplification we get 
  \[
  (t-1) \frac{ ( \delta_1 n )^s }{ s!}  
  \geq 
\frac{ \delta_2 n }{s! } \left(  \frac{ c_{1,2} n^{2-1/s} }{ \delta_2 n } - s \right)^s
+ 
\frac{ \delta_3 n }{s! } \left(  \frac{ c_{1,3} n^{2-1/s} }{ \delta_3 n } - s \right)^s.
  \]
For $j \in \{2,3 \}$, we can assume that $\dfrac{ c_{1,j} n^{2 - 1/s }}{ \delta_j n } > s $ otherwise 
\[
e(A_1 , A_j) = c_{1,j} n^{2 - 1/s} \leq s \delta_j n \leq s n  = o ( n^{2 - 1/s}).
\]  
From the inequality $(1 + x)^s \geq 1 + sx$ for $x \geq -1$, we now have
\begin{eqnarray*}
(t - 1) \delta_1^s n^s 
& \geq & 
\delta_2 n 
\left(  \frac{ c_{1,2} n^{2 - 1/s} }{ \delta_2 n } \right)^s - 
\delta_2 n s^2 \left( \frac{ c_{1,2} n^{2 - 1/s} }{ \delta_2 n } \right)^{s - 1} \\
& + &
\delta_3 n 
\left(  \frac{ c_{1,3} n^{2 - 1/s} }{ \delta_3 n } \right)^s - 
\delta_3 n 
s^2 \left( \frac{ c_{1,3} n^{2 - 1/s} }{ \delta_3 n } \right)^{s - 1}.
\end{eqnarray*}

Multiplying through by $n^{- s} \delta_2^{s-1} \delta_3^{s-1}$ and rearranging gives
\[
(t - 1) \delta_1^s \delta_2^{s - 1} \delta_3^{s - 1} 
\geq 
c_{1,2}^s \delta_3^{ s - 1} + c_{1,3}^s \delta_2^{s - 1} 
- 
\frac{s^2 \delta_3^{s - 1}\delta_2 c_{1,2}^{s - 1}  }{ n^{1 - 1/s} } 
- 
\frac{s^2 \delta_2^{s - 1}\delta_3 c_{1,3}^{s - 1} }{ n^{1 - 1/s} }.
\]

Since $\delta_2$ and $\delta_3$ are both at most 1 and $c_{1,j}$ is at most $\beta_{s,t}$, these last two terms 
are $o(1)$ (as $n$ goes to infinity) and so 
\[
(t - 1) \delta_1^s \delta_2^{s - 1} \delta_3^{s - 1} 
\geq 
c_{1,2}^s \delta_3^{ s - 1} + c_{1,3}^s \delta_2^{s - 1} - o(1).
\]
By symmetry between the parts $A_1$, $A_2$, and $A_3$, 
\[
(t - 1) \delta_2^s \delta_1^{s - 1} \delta_3^{s - 1} 
\geq 
c_{1,2}^s \delta_3^{ s - 1} + c_{2,3}^s \delta_1^{s - 1} - o(1) 
\]
and
\[
(t - 1) \delta_3^s \delta_1^{s - 1} \delta_2^{s - 1} 
\geq 
c_{1,3}^s \delta_2^{ s - 1} + c_{2,3}^s \delta_1^{s - 1} - o(1).
\]
Add these three inequalities together and divide by 2 to obtain 
\[
\frac{t -1}{2} \delta_1^{s - 1} \delta_2^{s - 1} \delta_3^{s - 1} ( \delta_1 + \delta_2 + \delta_3) 
\geq 
c_{1,2}^s \delta_3^{s-1} + c_{1,3}^s \delta_2^{s-1} + c_{2,3}^s \delta_1^{s - 1} -
o(1).
\]
Now $n = |A_1| + |A_2| + |A_3| = ( \delta_1 + \delta_2 + \delta_3 ) n $ so we may replace 
$\delta_1 + \delta_2 + \delta_3 $ with 1.  
This leads us to the optimization problem of maximizing 
\[
c_{1,2} + c_{1,3} + c_{2,3}
\]
subject to the constraints 
\begin{center}
$0 \leq \delta_i$, ~~~~~~~~~~ $0 \leq c_{i,j} \leq 1$,~~~~~~~~~~ $\delta_1 + \delta_2 + \delta_3 = 1$, 
\end{center}
and 
\[
\frac{t-1}{2} \delta_1^{s-1} \delta_2^{s-1} \delta_3^{s-1} 
\geq 
\delta_3^{s-1} c_{1,2}^s + \delta_2^{s-1} c_{1,3}^s + \delta_1^{s-1} c_{2,3}^s.  
\]
This can be done using the method of Lagrange Multipliers (see the Appendix) and gives  
\[
c_{1,2} + c_{1,3} + c_{2,3} \leq \left( \frac{1}{3} \right)^{1 - 1/s} \left( \frac{ t- 1}{2} \right)^{1/s}.
\]
We conclude that the number of edges of $G$ is at most 
\[
\left( \frac{1}{3} \right)^{1 - 1/s} \left( \frac{ t- 1}{2} \right)^{1/s} n^{2 - 1/s} + o( n^{2 - 1/s} ).
\]
\end{proof}

\bigskip

Now we prove Theorem \ref{new ub}.  First we recall some definitions from graph regularity.
Let $0 < p \leq 1$.  If $X$ and $Y$ are a pair of disjoint non-empty subsets of vertices in a graph $G$,
define $d_p (X,Y) = \frac{1}{p} d(X,Y)$ where 
\[
d(X,Y) = \frac{ e(X,Y) }{ |X| |Y| }
\]
is the density between $X$ and $Y$.  The pair 
$(X,Y)$ is \emph{$( \epsilon , p)$-regular} if 
\[
| d_p (X' , Y') - d_p (X,Y) | \leq \epsilon
\]
for all $X' \subseteq X$, $Y' \subseteq Y$ with $|X'| \geq \epsilon |X|$ and $|Y'| \geq \epsilon |Y|$.  

Suppose $V(G) = V_0 \cup V_1 \cup \dots \cup V_k$ is a partition of the vertex set of a graph $G$.
This partition is \emph{$(\epsilon , p )$-regular} if $|V_0| \leq \epsilon n$, 
$|V_1| = \dots = |V_k|$, and all but at most $\epsilon k^2$ of the pairs
$(V_i , V_j)$ with $1 \leq i , j \leq k$ are $(\epsilon , p )$-regular.  

Given $0 \leq d \leq 1$, the \emph{$(\epsilon , d , p)$-cluster graph} associated to a given
$(\epsilon , p)$-regular partition is the graph with vertex set $\{V_1, \dots , V_k \}$ (the parts of the partition
excluding $V_0$), and $\{ V_i , V_j \}$ is an edge if and only if $(V_i,V_j)$ is an 
$( \epsilon , p )$-regular pair with $d_p (V_i , V_j) \geq d$.  We will reserve the letter $R$ 
for an $( \epsilon ,d,p)$-cluster graph.  

Finally, Scott's Sparse Regularity Lemma tells us that $( \epsilon , p)$-regular partitions exist for any graph $G$ and,
crucially, the number of parts does not depend on the number of vertices of $G$.  

\begin{theorem}[Scott's Sparse Regularity Lemma]
Let $\epsilon > 0$ and let $C \geq 1$ be a constant.  There is an integer $T$, depending only on $\epsilon$, such that if 
$G$ is any graph with $e(G) \leq C pn^2$, then $G$ has an $(\epsilon , p )$-regular partition
where the number of parts is between $\epsilon^{-1} $ and $T$.  
\end{theorem}

\begin{proof}[Proof of Theorem \ref{new ub}]
Let $s =2$ and $ t \geq 2$, or let $s = t = 3$.  Define $( \alpha , \rho , p )$ by 
\[
( \alpha , \rho , p ) 
= 
\left\{
\begin{array}{ll}
(3/2 , \sqrt{t-1} , n^{-1/2} ) & \mbox{if $s =2 , t \geq 2$}, \\
(5/3 , 1 , n^{-1/3} ) & \mbox{if $s=t=3$.}
\end{array}
\right.
\]
This is the notation used in \cite{aksv}.  These parameters are chosen because for these 
particular values of $s$ and $t$,
\begin{center}
$z(m,n,K_{s,t}) = \sqrt{t-1} m n^{1/2} + o( mn^{1/2})$
~~
and ~~
$z(m,n,K_{3,3}) = m n^{2/3} + o(mn^{2/3})$
\end{center}
for $n \geq m$.  Fix a (small) positive constant $\gamma$.  
By Lemma 4.1 of \cite{aksv}, 
there is an $\epsilon_0 > 0$ and $d_0$ such that for any $0 < \epsilon \leq \epsilon_0$
and $0 < d \leq d_0$ and $T$, there is an $n_0$ such that 
the following holds.  If $G$ is any $n$-vertex $K_{s,t}$-free graph 
with $n \geq n_0$, and $R$ is an $(\epsilon ,d , p )$-cluster graph with $p = n^{ \alpha -2}$ obtained
from applying Scott's Sparse Regularity Lemma, then $R$ has $t$ vertices with $\epsilon^{-1} \leq t \leq T$.
Additionally, if 
\[
e(G) = ( \mu^{\alpha -1} + \gamma ) \rho p \frac{n^2}{2}
\]
where $\mu > 0$, then   
\begin{equation}\label{edges in R}
 e(R) \geq \left( \mu - \gamma \right) \frac{t^2}{2}
 \end{equation}
(this is the transference of density 
 mentioned after Theorem \ref{aksv theorem} in the Introduction).  
If we assume that $G$ is $k$-partite, then $R$ is also $k$-partite.  
 The number of edges in a 
 $k$-partite graph with $t$ vertices is at most $\binom{k}{2} \left( \frac{t}{k} \right)^2$ so 
 \begin{equation}\label{edges in R 2}
 e(R) \leq \binom{k}{2} \left( \frac{t}{k} \right)^2.
 \end{equation}
 Combining (\ref{edges in R}) and (\ref{edges in R 2}) gives $\mu \leq 1 - \frac{1}{k} + \gamma$.
 This upper bound on $\mu$ implies
 \[
 e(G) \leq \left( \left( \gamma + 1 - \frac{1}{k} \right)^{ \alpha -1 } + \gamma \right) \rho p \frac{n^2}{2}.
 \]
 When $s = 2$ and $t \geq 2$, we get
 \[
 e(G) \leq \left( \left( 1 - \frac{1}{k} \right)^{1/2} + o(1) \right) \frac{ \sqrt{t-1} }{2} n^{3/2} .
 \]
 When $s = t =3 $, 
 \[
 e(G) \leq \left( \left( 1 - \frac{1}{k} \right)^{2/3} + o(1)) \right) \frac{n^{5/3}}{2}.
 \]
  \end{proof}


\section{Proof of Theorem \ref{main theorem}}\label{section lb}

In this section we construct a 3-partite $K_{2,2t+1}$-free graph with many edges.
The construction is inspired by F\"{u}redi's construction of dense $K_{2,t}$-free graphs \cite{F3}.  

Let $t \geq 1$ be an integer.  Let $q$ be a power of a prime chosen so that 
$t$ divides $q - 1$ and let $\theta$ be a generator of the 
multiplicative group $\mathbb{F}_{q^2}^* := \mathbb{F}_{q^2} \backslash \{ 0 \}$.  
Let $A \subset \mathbb{Z}_{q^2 - 1}$ be defined by 
\[
A = \{ a \in \mathbb{Z}_{q^2 - 1} : \theta^a - \theta \in \mathbb{F}_q \}
\]    
and note that $|A| = q$.  
The set $A$ is sometimes called a \emph{Bose-Chowla Sidon} set and such sets were 
constructed by Bose and Chowla \cite{bose-chowla}.  
Let $H$ be the subgroup of $\mathbb{Z}_{q^2 - 1}$ generated 
by $(\frac{q-1}{t} ) (q + 1)$.   
Thus,  
\[
H = \left\{ 0 , 
\left( \frac{q-1}{t} \right) (q + 1) , 
2 \left(\frac{q-1}{t} \right) (q + 1) , \dots , 
(t-1) \left( \frac{q-1}{t} \right) (q + 1)
 \right\}.
\]
Note that $H$ is contained in the subgroup of $\mathbb{Z}_{q^2 - 1}$
generated by $q + 1$.  
Let $G_{q,t}$ be the bipartite graph whose parts are $X$ and $Y$ where each of $X$ and $Y$ is 
a disjoint copy of the quotient group $\mathbb{Z}_{q^2 - 1} / H$.    
A vertex $x + H \in X$ is adjacent to $x + a + H \in Y$ for all $a \in A$.

We will need the following lemma, which was proved in \cite{tt}.

\begin{lemma}\label{disjoint lemma}[Lemma 2.2 of \cite{tt}]
Let $A\subset \mathbb{Z}_{q^2-1}$ be a Bose-Chowla Sidon set. Then 
\[
A-A = \mathbb{Z}_{q^2-1} \setminus \{q+1, 2(q+1), 3(q+1), \ldots, (q-2)(q+1)\}.
\]
\end{lemma}
In particular, Lemma \ref{disjoint lemma} implies that $(A-A) \cap H = \emptyset$.

\begin{lemma}\label{new lemma 1}
If $t \geq 1$ is an integer and $q$ is a power of a prime for which 
$t$ divides $q - 1$, then the 
graph $G_{q,t}$ is a bipartite graph with $\frac{q^2 - 1}{t}$ vertices in 
each part, is $K_{2,t+1}$-free, and has $q \left( \frac{q^2 - 1}{t} \right)$ edges.
\end{lemma}
\begin{proof}
It is clear that $G_{q,t}$ is bipartite and has $\frac{q^2 - 1}{t}$ vertices in each 
part.  Let $x +H$ be a vertex in $X$.  The neighbors of $x+H$ are 
of the form $x + a + H$ where $a \in A$.  We now show that these vertices are all distinct.
If $x + a + H = x + b + H$ for some $a,b \in H$, then $a- b \in H$.  By Lemma \ref{disjoint lemma}
\[
(A - A) \cap H = \{ 0 \}
\]
where $A - A = \{ a - b : a , b \in A \}$.  
We conclude that $a = b$ and so the degree of $x+H$ is $|A| = q$.  
This also implies that $G_{q,t}$ has $q  \left( \frac{q^2 - 1}{t} \right)$ edges.  To finish the proof, 
we must show that $G_{q,t}$ has no $K_{2,t+1}$.

We consider two cases depending on which part contains the part of size two  
of the $K_{2,t+1}$.  First suppose that $x + H$ and $y+H$ are distinct vertices in $X$ and 
let $z +H$ be a common neighbor in $Y$.  Then $z + H = x + a + H$ and 
$z+H = y + b + H$ for some $a,b \in A$.  Therefore, 
$z = x + a + h_1$ and $z = y + b + h_2$ for some $h_1, h_2 \in H$.  
From this pair of equations we get $a - b= y - x + h_2 - h_1$.  
Since $H$ is a subgroup, $h_2 - h_1 = h_3$ for some $h_3 \in H$ and we have 
\begin{equation}\label{eq 1.1}
a- b = y - x + h_3.
\end{equation}  
The right hand side of (\ref{eq 1.1}) is not zero since $x+H$ and $y+H$ 
are distinct vertices in $A$.  
Because $A$ is a Sidon set
and $y - x + h_3 \neq 0$, there is at most one ordered pair $(a,b) \in A^2$ for 
which $a-b = y - x + h_3$.  There are $t$ possibilities for $h_3$ and so 
$t$ possible ordered pairs $(a , b) \in A^2$ for which 
\[
z+H = x+a + H =y+b+H
\]
is a common neighbor of $x+H$ and $y+H$.  This shows 
that $x+H$ and $y+H$ have at most $t$ common neighbors.

Now suppose $x +H$ and $y + H$ are distinct vertices in $Y$, and $z +H$ is a 
common neighbor in $X$.  There are elements $a,b \in A$ such that 
$z + a +H = x + H$ and $z + b + H = y + H$.  Thus, 
$z +a  + h_1 = x$ and $z + b + h_2 = y$ for 
some $h_1 , h_2 \in H$.   Therefore, 
$x - a - h_1 = y - b - h_2$ so $a - b = x - y + h_2 - h_1$.  We can then argue as before that 
there are at most $t$ ordered pairs $(a,b) \in A^2$ such that 
$z+H$ is a common neighbor of $z+a+H = x+H$ and 
$z+b+H = y + H$.  
 \end{proof} 

\bigskip

Once again, let $t \geq 1$ be an integer and let $q$ be a power of a prime for 
which $t$ divides $q - 1$.  
Let $\Gamma_{q,t}$ be the 3-partite graph with parts $X$, $Y$, and $Z$ where each 
part is a copy of the quotient group $\mathbb{Z}_{q^2 - 1} / H$.
Here $H$ is the subgroup generated by $( \frac{q-1}{t} ) (q+ 1)$.   
A vertex $x + H \in X$ is adjacent to $x + a + H \in Y$ for all $a \in A$.
Similarly, a vertex $y+ H \in Y$ is adjacent to $y + a + H \in Z$ for all $a \in A$, and 
a vertex $z + H \in Z$ is adjacent to $z + a + H \in X$ for all $a \in A$.

\begin{lemma}\label{new lemma 2}
The graph $\Gamma_{q,t}$ is $K_{2, 2t+1}$-free.
\end{lemma}
\begin{proof}
By Lemma \ref{new lemma 1}, a pair of vertices in one part of $\Gamma_{q,t}$ have 
at most $t$ common neighbors in each of the other two parts.  Thus, 
there cannot be a $K_{2,2t+1}$ in $\Gamma_{q,t}$ where the part of size two is contained 
in one part.  

Now let $x+H$ and $y +H$ be vertices in two different parts.
Without loss of generality, assume $x+H \in X$ and $y+H \in Y$.  
Suppose $z+H \in Z$ is a common neighbor of $x+H$ and $y+H$.  
There are elements $a,b \in A$ such that 
$z +  H = y + a +H$ and $z + b + H = x + H$, so we have 
\begin{center}
$z = y + a + h_1$ ~~ and ~~ $z + b = x + h_2$
\end{center}
for some $h_1 , h_2 \in H$.  This pair of equations implies
\[
a + b = x - y + h_2 - h_1.
\]
Since $H$ is a subgroup, $h_2 - h_1 \in H$.  Let $h_2  -h_1 = h_3$ where $h_3 \in H$
so 
\[
a+b = x - y + h_3.
\]
There are $t$ possibilities for $h_3$.  Given $h_3$, the 
equation $a + b = x - y + h_3$ uniquely determines the pair $\{ a , b \}$ since $A$ is a Sidon set.
There are two ways to order $a$ and $b$ and so 
$x+H$ and $y+H$ have at most $2t$ common neighbors in $Z$. 
 \end{proof} 

\bigskip

\begin{proof}[Proof of Theorem \ref{main theorem}]
By Theorem \ref{chi 3 ub}, 
\[
\textup{ex}_{ \chi \leq 3 } ( n , K_{2 , 2t+1} ) = \sqrt{ \frac{1}{3} } \left(   \frac{ 2t + 1 - 1 }{2} + o(1)) \right)^{1/2} n^{3/2}
= \sqrt{ \frac{t}{3}} n^{3/2} + o(n^{3/2}).
\]
As for the lower bound, if $q$ is any power of a prime for which $t$ divides $q - 1$, then by Lemmas \ref{new lemma 1} 
and \ref{new lemma 2}, the graph 
$\Gamma_{q,t}$ is a 3-partite graph with $\frac{q^2 - 1}{t}$ vertices in each part, is $K_{2,2t+1}$-free, 
and has $3 q  \left( \frac{q^2 - 1}{t} \right)$ edges.  Thus,
\[
\textup{ex}_{ \chi \leq 3} \left(  \frac{3(q^2 - 1)}{t} , K_{2 , 2t+1} \right) \geq 3q  \left( \frac{q^2 - 1}{t} \right).
\]
If $n = \frac{ 3 (q^2 - 1) }{t}$, then the above can be rewritten as 
\[
\textup{ex}_{ \chi \leq 3} ( n , K_{2 , 2t+1}) \geq n \left( \sqrt{ \frac{nt}{ 3} +1}  \right) \geq \sqrt{ \frac{t}{3} } n^{3/2} - n.
\]
A standard density of primes argument finishes the proof.  
\end{proof}


\section{Concluding Remarks}\label{conclusion}
We may consider a similar graph to $G_{q,t}$ and $\Gamma_{q,t}$ which does not necessarily have bounded chromatic number. Let $\Gamma$ be a finite abelian group with a subgroup $H$ of order $t$. Let $A\subset \Gamma$ be a Sidon set such that $(A-A) \cap H = \{0\}$. Then we may construct a graph $G$ with vertex set $\Gamma / H$ where $x+H$
is adjacent to $y + H$ if and only if $x+y = a+h$ for some $a\in A$ and $h\in H$. 
The proof of Lemma \ref{new lemma 1} shows that $G$ is a $K_{2,t+1}$-free graph on $|\Gamma|/|H|$ vertices and every vertex has degree $|A|$ or $|A|-1$.

When $\Gamma = \mathbb{Z}_{q^2-1}$, $t$ divides $q-1$, and $A$ is a Bose-Chowla Sidon set, the resulting graph $G$ is similar to the one constructed by F\"uredi in \cite{F3}.  In general, these graphs may or may not be isomorphic and some 
computational results suggest these graphs are isomorphic when $q \equiv 1 ( \textup{mod}~4)$.  
For example, when $q=19$ and $t\in \{1,2,3,6\}$ the graph constructed above has one more edge than the graph constructed by F\"uredi.
However, when $q = 17$ and $t \in \{1,2, 4 \}$, the graphs are isomorphic.  

Turning to the question of determining $\mathrm{ex}_{\chi\leq 3}(n, C_4)$, Theorem \ref{chi 3 ub} shows that 
\[
\mathrm{ex}_{\chi\leq 3}(n, C_4) \lesssim \frac{n^{3/2}}{\sqrt{6}}. 
\]
Furthermore, the optimization shows that if this bound is tight asymptotically, then a construction would have to be $3$-partite with each part of size asymptotic to $\frac{n}{3}$, and average degree asymptotic to $\sqrt{\frac{n}{6}}$ between each part. The following construction is due to Jason Williford \cite{W}.

\begin{theorem}
Let $R$ be a finite ring, $A\subset R$ an additive Sidon set and 
\[
B = cA = \{c a:a\in A\}.
\] 
If $(A-A) \cap (B-B) = \{0\}$ where $c$ is invertible, then 
there is a graph on $3|R|$ vertices which is $3$-partite, $C_4$-free and is $|A|$-regular between parts.
\end{theorem}
\begin{proof}
We construct a graph with partite sets $S_1, S_2, S_3$ where $S_1 = R$, $S_2 = \{A+i\}_{i\in R}$ and $S_3 = \{B+j\}_{j\in R}$. A vertex in $S_1$ is adjacent to a vertex in $S_2$ or $S_3$ by inclusion. 
The vertex $A+j\in S_2$ is adjacent to $B+i \in S_3$ if $-cj + i\in A$. Since $c$ is invertible, we have that both $A$ and $B$ are Sidon sets.
Therefore, the bipartite graphs between $S_1$ and $S_2$, and between $S_1$ and $S_3$ are incidence graphs of partial linear spaces, and thus do not contain $C_4$. 

If there were a $C_4$ with $A+i, A+j\in S_2$ and $B+k, B+l \in S_3$, it implies that there exist $a_1,a_2,a_3,a_4\in A$ such that
\begin{align*}
-ci + k & = a_1\\
-ci+l & = a_2\\
-cj + k & = a_3\\
-cj + l &=a_4.
\end{align*}
This means that $k-l = a_1-a_2 = a_3-a_4$. Since $A$ is a Sidon set this means that $a_1 = a_2$ or $a_1 = a_3$, which implies that either $k=l$ or $i = j$.

If there were a $C_4$ with $i \in S_1$, $A+j, A+k\in S_2$, and $B+l\in S_3$, then there are $a_1,a_2,a_3,a_4\in A$ such that 
\begin{align*}
i&=a_1+j\\
i&=a_2+k\\
-cj+l &=a_3\\
-ck+l &=a_4.
\end{align*}
Thus, $c(j-k) = c(a_2-a_1) = a_4-a_3$. Since $B=cA$ we have that $b_2-b_1 = a_4-a_3$ for some $b_1,b_2\in B$, and therefore $b_2-b_1 = a_4-a_3 = 0$. This implies that $j=k$. The case when there are two vertices in $S_3$ and one each in $S_1$ and $S_2$ is similar.
\end{proof}
\medskip

The condition that $(A-A) \cap (B-B) = \{0\}$ and $A$ is a Sidon set implies that $2|A|(|A|-1) \leq |R| -1$. In $\mathbb{Z}_5$, 
if $A = \{0,1\}$ and $B = 2A = \{0,2\}$, we have $(A-A) \cap (B-B) = \{0\}$ and $(A-A)\cup (B-B) = \mathbb{Z}_5$. This gives a $3$-partite graph on $15$ vertices which is $C_4$-free and is $4$-regular. In $\mathbb{Z}_{41}$, the set $A = \{1,10,16,18,37\}$ and $B = 9A$ have the same property that $(A-A) \cap (B-B) = \{0\}$ and $(A-A) \cup (B-B) = \mathbb{Z}_{41}$. This gives a $3$-partite $C_4$-free graph on $123$ vertices which is $10$ regular.
These two lower bounds, together with inequality (\ref{special inequality}) from the proof of Theorem \ref{chi 3 ub} show
that 
\begin{center}
$\textup{ex}_{\chi \leq 3} (15 , C_4) =30$ ~~~ and ~~~ $\textup{ex}_{\chi \leq 3} ( 123 , C_4) = 615$.
\end{center}

In general, a $(v,k,\lambda)$-difference family in a group $\Gamma$ of order $v$ is a collection of sets $\{D_1,\ldots, D_t\}$,
 each of size $k$, such that the multiset 
\[
(D_1-D_1) \cup \cdots \cup (D_t - D_t)
\]
contains every nonzero element of $\Gamma$ exactly $\lambda$ times. If one could find an infinite family of $(2k^2-2k+1, k, 1)$-difference families in $\mathbb{Z}_{2k^2-2k+1}$ where the two blocks are multiplicative translates of each other by a unit, then the resulting graph would match the upper bound in Theorem \ref{chi 3 ub}. The sets $A=\{0,1\}$ and $2A$ in $\mathbb{Z}_5$, and $A=\{1,10,16,18,37\}$ and $9A$ in $\mathbb{Z}_{41}$ are examples of this for $k=2$ and $k=5$, respectively. We could not figure out how to extend this construction in general.  In \cite{cz} it is shown that no $(61, 6, 1)$-difference family exists in $\mathbb{F}_{61}$.

To show Theorem \ref{chi 3 ub} is tight asymptotically it would suffice to find something weaker than a $(2k^2-2k+1, k, 1)$-difference family where the two blocks are multiplicative translates of each other. We do not need every nonzero element of the group to be represented as a difference of two elements, just a proportion of them tending to $1$.

\section{Acknowledgements}

The authors would like to thank Casey Tompkins for introducing the first author to the problem.  
We would also like to thank Cory Palmer for helpful discussions.



\section{Appendix}\label{appendix}

Here we solve the optimization problem of Theorem \ref{chi 3 ub} using the method of 
Lagrange Multipliers.  For convenience, we write $x$ for $c_{1,2}$, $y$ for $c_{1,3}$, and 
$z$ for $c_{2,3}$.  Recall that $\delta_1$, $\delta_2$, and $\delta_3$ are positive real numbers
that satisfy $\delta_1 + \delta_2 + \delta_3 =1$.  
Let 
\[
f(x , y , z) = x + y + z
\]
and 
\[
g(x , y , z) = \frac{t-1}{2} \delta_1^{s-1} \delta_2^{s-1} \delta_3^{s-1} 
- \delta_3^{s-1} x^s - \delta_2^{s-1} y^s - \delta_1^{s-1} z^s.
\]
For a parameter $\lambda$, let $L (x,y,z, \lambda ) = f(x,y,z) + \lambda g(x,y,z)$.
Taking partial derivatives, we get 
\begin{equation}\label{appendix 1}
L_x = 1 - s \lambda \delta_3^{s-1} x^{s-1} = 0, 
\end{equation}
\begin{equation}\label{appendix 2}
L_y= 1 - s \lambda \delta_2^{s-1} y^{s-1} = 0, 
\end{equation}
\begin{equation}\label{appendix 3}
L_z = 1 - s \lambda \delta_1^{s-1} z^{s-1} = 0, 
\end{equation}
\begin{equation}\label{appendix 4}
\lambda \left( \frac{t-1}{2} \delta_1^{s-1} \delta_2^{s-1} \delta_3^{s-1}
 - \delta_3^{s-1} x^s - \delta_2^{s-1} y^s - \delta_1^{s-1} z^s \right) = 0. 
\end{equation}
Note that $\lambda \neq 0$ otherwise we contradict (\ref{appendix 1}) so by (\ref{appendix 4}), 
\begin{equation}\label{appendix 5}
\frac{t-1}{2} \delta_1^{s-1} \delta_2^{s-1} \delta_2^{s-1}
 =
  \delta_3^{s-1} x^s + \delta_2^{s-1} y^s + \delta_1^{s-1} z^s.
\end{equation}
From (\ref{appendix 1}), (\ref{appendix 2}), and (\ref{appendix 3}) we have 
\begin{equation}\label{x y z equation}
\left( \frac{1}{ 2 \lambda } \right)^{ \frac{1}{s - 1}} = \delta_3 x = \delta_2 y = \delta_1 z.
\end{equation}
Combining this with (\ref{appendix 5}) and using $\delta_3 = 1 - \delta_1 - \delta_2$, we get an equation 
that can be solved for $x$ to obtain 
\[
x = 
\left(   \frac{ (t - 1) \delta_1^s \delta_2^s }{   2 ( \delta_1 ( 1 - \delta_1 ) + \delta_2 ( 1 - \delta_2) - \delta_1 \delta_2 ) } 
\right)^{1/s}.
\]
Using (\ref{x y z equation}), we can then solve for $y$ and $z$ and get
\[
x + y + z 
= 
\frac{  (t-1)^{1/s} }{ 2^{1/s} } \left( \delta_1 ( 1 - \delta_1 ) + \delta_2 ( 1 - \delta_2) - \delta_1 \delta_2 \right)^{1 - 1/s}.
\]
The maximum value of 
\[
\delta_1 ( 1 - \delta_1 ) + \delta_2 ( 1 - \delta_2) - \delta_1 \delta_2
\]
over all 
$\delta_1 , \delta_2 \geq 0$ for which $0 \leq \delta_1 + \delta_2 \leq 1$ is $\frac{1}{3}$ and it is obtained 
only when $\delta_1 = \delta_2 = \frac{1}{3}$.  
Therefore, 
\[
x + y + z \leq \frac{  (t - 1)^{1/s} }{ 2^{1/s} } \left( \frac{1}{3} \right)^{1 - 1/s}.
\]

\end{document}